\documentclass[11pt]{amsart}
\usepackage{amscd,amssymb,graphics}
\usepackage{hyperref}
\usepackage{amsfonts}
\usepackage{amsmath}
\usepackage{enumerate}
\input xy
\xyoption{all}

\newtheorem{thm}{Theorem}[section]

\newtheorem{lemma}[thm]{Lemma}

\newtheorem{quest}[thm]{Question}
\newtheorem{defi}[thm]{Definition}

\theoremstyle{remark}
\newtheorem{remark}[thm]{Remark}

\newcommand{\ben}{\begin{enumerate}}
\newcommand{\een}{\end{enumerate}}
\newcommand{\bit}{\begin{itemize}}
\newcommand{\eit}{\end{itemize}}

\def\QED{\nobreak\quad\ifmmode\roman{Q.E.D.}\else{\rm Q.E.D.}\fi}

\def\sna{\scriptscriptstyle\mathcal{NA}}

\begin{document}
\title[]{On Graev type ultra-metrics}

\author{Menachem Shlossberg}
\address{Department of Mathematics,
Bar-Ilan University, 52900 Ramat-Gan, Israel}
\email{shlosbm@macs.biu.ac.il}



\begin{abstract} We study Graev ultra-metrics which were introduced by Gao
\cite{GAO}. We show that the free non-archimedean balanced
topological group defined over an ultra-metric space is metrizable
by a Graev ultra-metric. We prove that the Graev ultra-metric has a
maximal property.  Using this property, among others,  we show that
the Graev ultra-metric associated with an ultra-metric space $(X,d)$
with diameter$\leq 1$
 coincides with the ultra-metric $\hat{d}$
of Savchenko and  Zarichnyi  \cite{SZ}.

\end{abstract}

\maketitle

\section{\bf Introduction and Preliminaries}
  A uniform space is {\it non-archimedean} if it has a
base of equivalence relations. A metric $d$ is called {\it
ultra-metric} if it satisfies the strong triangle inequality.
Clearly, the metric uniformity of every ultra-metric space $(X,d)$
is non-archimedean.
 By  Graev's  Extension Theorem (see \cite{GRA}),
  for every metric $d$ on  $X\cup \{e\}$ there exists a metric $\delta$ on the free group $F(X)$ with the following
properties: \begin{enumerate}
\item $\delta$ extends $d.$
\item $\delta$ is a two sided invariant metric on $F(X).$
\item $\delta$ is maximal among all invariant metrics on $F(X)$
extending $d.$ \end{enumerate}

 Gao \cite{GAO} has  recently presented the notion of {\it Graev
ultra-metric}, a natural ultra-metric modification  to Graev's
classical construction.
 We study  this relatively new concept, after reviewing
it in this section.  In Section \ref{sub:max} we show that Graev
ultra-metrics satisfy a maximal property (Theorem
\ref{theorem:grau}).  Recall that according to \cite{Mar} any
continuous map from a Tychonoff space $X$ to a topological group $G$
can be uniquely extended to a continuous homomorphism from the
(Markov) free topological group $F(X)$ into $G.$ Free topological
groups were studied by  researchers in different contexts. See for
example, \cite{AT, Tk, pes85, Us-free, Sip, Me-F,Num2,Pest-Cat,Mor}.
 In Section \ref{sub:bal} we show that the
uniform free non-archimedean balanced topological group  defined
over an ultra-metric space is metrizable by a Graev ultra-metric
(Theorem
 \ref{thm:metgr}). In Section \ref{sub:zar} we compare between seemingly different
ultra-metrics that are defined on the free group $F(X)$ (Theorem
\ref{thm:coinc}).
  We start with relevant notations and definitions from  \cite{GAO}. Considering a nonempty set $X$ we  define  $\overline{X}=X\cup X^{-1}\cup \{e\}$
where $X^{-1}=\{x^{-1}:\ x\in X\}$ is a disjoint copy of $X$ and
$e\notin X\cup X^{-1}.$ We agree that $(x^{-1})^{-1}=x$ for every
$x\in X$ and also that $e^{-1}=e.$ Let $W(X)$ be the set of words
over the alphabet $\overline{X}$.

We call a word $w\in W(X)$ {\it irreducible} if  either one of the
following conditions holds:\begin{itemize} \item $w=e$
\item $w=x_0\cdots x_n$ does not contain the letter $e$ or a sequence
of two adjacent letters of the form $xx^{-1}$ where $x\in X\cup
X^{-1}.$ \end{itemize} The length of a word $w$ is denoted by
$lh(w).$ $w'$ is the {\it reduced word} for $w\in W(X).$ It is the
irreducible word obtained from $w$ by applying repeatedly the
following algorithm: replace any appearance of $xx^{-1}$ by $e$ and
eliminate $e$ from any occurrence of the form $w_1ew_2,$ where at
least one of $w_1$ and $w_2$ is nonempty. A word $w\in W(X)$ is {\it
trivial} if $w'=e.$ Now, as a set  the free group $F(X)$ is simply
the collection of all irreducible words. The group operation is
concatenation of words followed by word reduction. Note that the
identity element of $F(X)$ is $e$ and not the empty word.

\begin{defi}\label{def:grau}\cite[Definition 2.1]{GAO} Let $d$ be an ultra-metric on $\overline{X}$ for which
the following conditions hold for every $x,y\in X$:
\begin{enumerate}\item $d(x^{-1},y^{-1})=d(x,y).$
\item $d(x,e)=d(x^{-1},e).$
\item $d(x^{-1},y)=d(x,y^{-1}).$
\end{enumerate}   For $w=x_0\cdots x_n, \ v=y_0\cdots y_n\in W(X)$ put
$$\rho_u(w,v)=\max\{d(x_i,y_i): \ 0\leq i\leq n\}.$$

The \it{Graev ultra-metric} $\delta_u$ on $F(X)$ is defined as
follows:
$$\delta_u(w,v)=\inf\{\rho_u(w^\ast,v^{\ast}):\ w^{\ast},v^{\ast}\in
W(X), lh(w^{\ast})=lh(v^{\ast}), (w^{\ast})'=w, (v^{\ast})'=v\}, \
 $$ for every $w,v\in F(X).$
\end{defi}
The following concepts have a lot of significance in studying Graev
ultra-metrics.
\begin{defi}\cite{DG,GAO}\label{def:dag}
Let $m,n\in \mathbb N$  and $m\leq n.$ A bijection $\theta$ on
$\{m,\ldots, n\}$ is a {\it match} if \begin{enumerate}
\item $\theta\circ \theta=id$ and
\item there are no $ m\leq i,j\leq n$ such that $i<j<\theta(i)<\theta(j).$
\end{enumerate} For any match $\theta$ on $\{0,\ldots, n\}$ and $w=x_0\cdots
x_n\in W(X)$  define

$$x_{i}^\theta=\left\{
               \begin{array}{ll}
                 x_i, & \text{if}\ \theta(i)>i  \\
                 e, & \text{if}\ \theta(i)=i \\
                 x_{\theta(i)}^{-1}, & \text{if}\ \theta(i)<i
               \end{array}
             \right.$$

and $w^\theta=x_{0}^{\theta} \cdots x_{n}^{\theta}.$
\end{defi}

\begin{thm}\label{thm:gao} \begin{enumerate} \item \cite[Theorem
2.3]{GAO}  For any $$w\in F(X), \ \delta_u(w,e)=\min
\{\rho_u(w,w^{\theta}):\ \theta  \ \text{is a match}\}.$$
\item \cite[Theorem
2.4]{GAO} Let $(X, d)$ be an ultra-metric space. Then the Graev
ultra-metric $\delta_u$ is a two-sided invariant ultra-metric on
$F(X)$ extending $d$. Furthermore, $F(X)$ is a topological group in
the topology induced by $\delta_u$. If $X$ is separable, so is
$F(X)$. \end{enumerate}
\end{thm}

\section{A maximal property of Graev ultra-metrics}\label{sub:max}
Recall that given a metric $d$ on $X\cup \{e\},$ its associated
Graev metric is the maximal among all invariant metrics on $F(X)$
extending $d.$ This fact leads for a natural question:
\begin{quest}
Is Graev ultra-metric maximal in any sense?
\end{quest}

The following theorem provides a positive answer.
\begin{thm}\label{theorem:grau}
 Let $d$ be an ultra-metric on $\overline{X}$ for which
the following conditions hold for every $x,y\in X$:
\begin{enumerate}\item $d(x^{-1},y^{-1})=d(x,y).$
\item $d(x,e)=d(x^{-1},e).$
\item $d(x^{-1},y)=d(x,y^{-1}).$
\end{enumerate}     Then: \begin{enumerate} [(a)]
\item The Graev ultra-metric $\delta_u$ is maximal among all
invariant ultra-metrics on $F(X)$ that extend the metric $d$ defined
on $\overline{X}.$
\item If in addition $d(x^{-1},y)=d(x,y^{-1})=\max\{d(x,e),d(y,e)\}$
then $\delta_u$  is  maximal among all invariant ultra-metrics on
$F(X)$ that extend the metric $d$ defined on $X\cup
\{e\}.$\end{enumerate}
\end{thm}

\begin{proof}
 We prove (a) using the following claim.  \vskip 0.2cm \textbf{Claim 1:}  Let $R$ be an invariant ultra-metric on
$F(X)$ that extends the metric $d$ defined on $\overline{X}$ and
$w=x_0\cdots x_n\in F(X).$ Then for every match $\theta$ on
$\{0,\ldots, lh(w)-1\}$ we have
$$\rho_u(w,w^{\theta})\geq R(w,e).$$
\begin{proof} We prove the claim by induction on $lh(w).$ If $lh(w)=1$ then
the only match is the identity. In this case by definition
$w^\theta=e$ and also $w\in \overline{X}$ so
$$\rho_u(w,w^\theta)=\rho_u(w,e)=d(w,e)=R(w,e).$$

If $lh(w)=2$ then $w=x_{0}x_{1}$ where $x_0,x_1\in \overline{X}$ and
there are only two matches to consider: the identity map and a
transposition.

If $\theta=Id$ then
$$\rho_u(w,w^\theta)=\max\{{d}(x_0,e),(x_1,e)\}=$$
$$=\max\{R(x_0,e),R(x_1,e)\}=\max\{R(x_{0}x_{1},x_1),R(x_1,e)\}\geq R(w,e).$$

If $\theta$ is a transposition we have
$$\rho_u(w,w^\theta)={d}(x_1,x_{0}^{-1})=R(x_1,x_{0}^{-1})=R(w,e).$$

We can now  assume that $lh(w)\geq 3$
 and also  that the assertion is true for every word $t$ with
$lh(t)<lh(w).$ Let $\theta$ be a match on $\{0,\ldots, lh(w)-1\}$
(where $w=x_0\cdots x_n$ and $lh(w)=n+1$).

First case: $\theta(0)\neq n.$ In this case there exists $j\geq 1$
such that $\theta(j)=n.$  For every $j\leq i\leq n$ we have
$j\leq\theta(i)\leq n.$ Indeed, otherwise $j>\theta(i).$ Now,
$\theta(j)=n,\ \theta(n)=j$ so we conclude that $i\neq j$ and $i\neq
n.$ Therefore, $\theta(i)<j<i<n$ and we obtain that
$$\theta(i)<j<i<\theta(j),$$  contradicting the definition of a match.
This implies that   $\theta$  induces two matches: $\theta_1$ on
$\{0,\ldots, j-1\}$  and  $\theta_2$ on $\{j,\ldots n\}.$  Let
$g_1=x_0\cdots x_{j-1}, g_2=x_j\cdots x_n.$

Clearly $w=g_1g_2$ and using the induction hypothesis we obtain that
$$\rho_u(w,w^\theta)=\max\{\rho_u(g_1,g_{1}^{\theta_1}),\rho_u(g_2,g_{2}^{\theta_2})\}\geq \max\{R(g_1,e),R(g_2,e)\}=$$
$$ =\max\{R(g_{1}g_2,g_2),R(g_2,e)\}\geq
R(g_{1}g_{2},e)=R(w,e).$$

Second case: $\theta(0)=n$  where $n\geq 2.$ Then,
$$R(x_0\cdots x_n,e)=R(x_1\cdots
x_{n-1},x_{0}^{-1}x_{n}^{-1})\leq$$$$ \leq\max\{R(x_1\cdots
x_{n-1},e),R(x_{0}^{-1}x_{n}^{-1},e)\}=$$$$=
\max\{R(x_{0}x_n,e),R(x_1\cdots x_{n-1},e)\}.$$ Letting
$g_1=x_{0}x_n,\ g_2=x_1\cdots x_{n-1},$ we have $$R(w,e)\leq
\max\{R(g_1,e),R(g_2,e)\}.$$

Now, $\theta$ induces two matches on $\{0,n\}$ and on $\{1,\ldots,
n-1\}$ which we denote by $\theta_1,\theta_2$ respectively. From the
inductive step and also from the fact that the assertion is true for
words of length $2$ we have:
 $R(g_1,e)=R(x_{1}x_n,e)\leq \rho_u(g_1,g_{1}^{\theta_{1}})$
and also $R(g_2,e)\leq \rho_u(g_2,g_{2}^{\theta_{2}}).$
 On the one hand,
$$\rho_u(w,w^{\theta})=\max\{\rho_u(x_0,x_{0}^{\theta}),\rho_u(x_n,x_{n}^{\theta}),\rho_u(g_2,g_{2}^{\theta_{2}})\}=$$$$
=\max\{\rho_u(x_0,x_0),\rho_u(x_n,x_{0}^{-1}),\rho_u(g_2,g_{2}^{\theta_{2}})\}.$$

On the other hand,
$\rho_u(g_1,g_{1}^{\theta_1})=\max\{\rho_u(x_0,x_0),\rho_u(x_n,x_{0}^{-1})\}.$

Hence,
$$\rho_u(w,w^\theta)=\max\{\rho_u(g_i,g_{i}^{\theta_i}): 1\leq i\leq
2\}\geq$$$$\geq \max\{R(g_1,e),R(g_2,e)\}\geq R(w,e).$$
\end{proof}
To prove $(a)$ let $R$ be an invariant ultra-metric on $F(X)$ which
extends the metric $d$ defined on $\overline{X}.$ By the invariance
of both $\delta_u$ and $R$ it suffices to show that
$\delta_u(w,e)\geq R(w,e) \ \forall w\in F(X).$ The proof now
follows  from Theorem \ref{thm:gao}.1
 and Claim $1.$ The proof of $(b)$
is quite similar. It follows from the obvious analogue of Claim $1.$
We mention few necessary changes and observations in the proof. Note
that this time $R$ is an invariant ultra-metric on $F(X)$ which
extends the metric $d$ defined on $X\cup \{e\}.$ We have
$d(x,e)=R(x,e)\ \forall x\in \overline{X}.$ This is due to the
invariance of $R$ and  the equality $d(x,e)=d(x^{-1},e)$. This
allows us to use the  same arguments, as in the proof of Claim $1,$
to prove the cases $lh(w)=1$ and $lh(w)=2$ where $\theta=id.$ For
the case $lh(w)=2$ where $\theta$ is a transposition note that we do
not necessarily  have $d(x_1,x^{-1}_{0})=R(x_1,x^{-1}_{0}).$
However, by the additional assumption we do have
$$d(x_1,x^{-1}_{0})\geq R(x_1,x^{-1}_{0}).$$ Indeed,
$$d(x_1,x^{-1}_{0})=\max\{d(x_1,e),d(x_0,e)\}$$$$=\max\{R(x_1,e),R(x_0,e)\}=\max\{R(x_1,e),R(x^{-1}_{0},e)\}$$$$\geq
R(x_1,x^{-1}_{0}).$$ So, the assertion is true for $lh(w)=2.$
  The inductive step is left unchanged.
\end{proof}
\section{Uniform free non-archimedean balanced groups} \label{sub:bal}
\begin{defi} A topological group is: \begin{enumerate} \item
{\it non-archimedean} if it has a base at the identity consisting of
open subgroups. \item  {\it balanced} if its left and right
uniformities coincide.\end{enumerate}\end{defi}
 In \cite{MS} we proved that the free non-archimedean
balanced group of an ultra-metrizable uniform space is metrizable.
Moreover, we claimed that this group is metrizable by a Graev type
ultra-metric. In this section  we prove the last assertion in full
details (see  Theorem \ref{thm:metgr}). For the reader's convenience
we review the definition of this topological group and some of its
properties (see \cite{MS} for more details). For a topological group
$G$ denote by $N_e(G)$  the set of all neighborhoods at the identity
element $e$.
\begin{defi} \label{d:FreeGr} Let
 $(X,{\mathcal U})$ be a non-archimedean uniform space. The \emph{uniform free non-archimedean balanced
 topological group of $(X,{\mathcal U})$} is denoted by $F^b_{\sna}$ and  defined as follows:
 $F^b_{\sna}$ is a  non-archimedean balanced topological group for
 which
 there exists a universal uniform map
 $i: X \to F^b_{\sna}$
satisfying the following universal property. For every uniformly
continuous map \textbf{$\varphi: (X,{\mathcal U}) \to G$} into a
balanced non-archimedean topological group $G$ there exists a unique
continuous homomorphism \textbf{$\Phi: F^b_{\sna} \to G $} for which
the following diagram commutes:
\begin{equation*} \label{equ:ufn}
\xymatrix { (X,{\mathcal U}) \ar[dr]_{\varphi} \ar[r]^{i} &
F^b_{\sna}
\ar[d]^{\Phi} \\
  & G }
\end{equation*}
\end{defi}

Let $(X,{\mathcal U})$ be a non-archimedean uniform space,
$Eq(\mathcal U)$ be the set of equivalence relations from ${\mathcal
U}$.
Define two functions from $X^{2}$ to  $F(X):$ $j_2$ is the mapping
$(x,y)\mapsto x^{-1}y$  and  $j_{2}^{\ast}$ is the mapping
$(x,y)\mapsto xy^{-1}.$
\begin{defi}\cite[Definition 4.9]{MS}
 \label{def:desc} \begin{enumerate} \item Following \cite{pes85}, for  every $\psi\in {\mathcal U}^{F(X)}$
let $$V_{\psi}:=\bigcup_{w\in F(X)}w(j_{2}(\psi(w))\cup
j_{2}^{\ast}(\psi(w)))w^{-1}.$$

\item  As a particular case in which every $\psi$ is a constant
function we obtain the  set
$$\tilde{\varepsilon}:=\bigcup_{w\in F(X)}w(j_{2}(\varepsilon)\cup
j_{2}^{\ast}(\varepsilon))w^{-1}.$$ \end{enumerate}
\end{defi}
\begin{remark}\cite[Remark 4.10]{MS}\label{rem:sym}
Note that if $\varepsilon\in Eq(\mathcal U)$  then
$(j_{2}(\varepsilon))^{-1}=j_{2}(\varepsilon), \
(j_{2}^{\ast}(\varepsilon))^{-1}= j_{2}^{\ast}(\varepsilon)$ and
$$\tilde{\varepsilon}=\bigcup_{w\in F(X)}w(j_{2}(\varepsilon)\cup
j_{2}^{\ast}(\varepsilon))w^{-1}=\bigcup_{w\in
F(X)}wj_{2}(\varepsilon)w^{-1}.$$ Indeed, this follows from the
equality $wts^{-1}w^{-1}=(ws)s^{-1}t(ws)^{-1}.$ Note also that the
subgroup $[\widetilde{\varepsilon}]$  generated by $\varepsilon$ is
normal in $F(X)$.
\end{remark}

\begin{thm}\cite[Theorem 4.13.2]{MS}\label{thm:nafin}
Let $(X,{\mathcal U})$ be non-archimedean and let %
$\mathcal{B}\subseteq Eq(\mathcal U)$ be a  base of ${\mathcal U}$.

Then: \begin{enumerate}
\item the  family (of normal subgroups)
$\{[\tilde{\varepsilon}]:\ \varepsilon\in  \mathcal{B}\}$ is a base
of $N_{e}(F^b_ {\sna}).$
\item
the topology of $F^b_ {\sna}$ is the weak topology generated by the
system of homomorphisms $\{\overline{f_{\varepsilon}}: F(X) \to
F(X/\varepsilon)\}_{\varepsilon \in \mathcal{B}}$ on discrete groups
$F(X/\varepsilon)$. \end{enumerate}
\end{thm}
It follows from Theorem \ref{thm:nafin} that $F^b_ {\sna}$ is
metrizable if the  uniform space $(X,{\mathcal U})$ is metrizable.
In fact, in this case  $F^b_ {\sna}$ is metrizable by a Graev type
ultra-metric as the following theorem suggests.
\begin{thm} \label{thm:metgr}Let $(X,d)$ be an ultra-metric space. \begin{enumerate} \item Fix $x_0\in X$ and extend
the definition of $d$ from $X$ to $X':=X\cup \{e\}$ by letting
$d(x,e)=\max\{d(x,x_0),1\}.$ Next, extend it to $\overline{X}:=X\cup
X^{-1}\cup\{e\}$ by defining for every $x,y\in X\cup \{e\}:$
\begin{enumerate}
\item $d(x^{-1},y^{-1})=d(x,y) $
\item $d(x^{-1},y)=d(x,y^{-1})=\max\{d(x,e),d(y,e)\}$
\end{enumerate}
Then for $\varepsilon<1$ we have
$B_{\delta_u}(e,\varepsilon)=[\widetilde{\mathcal{E}}]$ where
$\delta_u$ is the Graev ultra-metric associated with $d$ and
$$\mathcal{E}:=\{(x,y)\in X\times X:d(x,y)<\varepsilon\}.$$
\item $F^b_ {\sna}(X,d)$ is metrizable by the Graev ultra-metric
associated with $(X,d).$
\end{enumerate}
\end{thm}
\begin{proof}
$(1):$  We first show that $ [\widetilde{\mathcal{E}}]\subseteq
B_{\delta_u}(e,\varepsilon).$ Since the open ball
$B_{\delta_u}(e,\varepsilon)$ is a normal subgroup of $F(X)$ it
suffices to show by (Remark \ref{rem:sym}) that
$j_2(\mathcal{E})\subseteq B_{\delta_u}(e,\varepsilon).$ Assuming
that $d(x,y)<\varepsilon$ we have
$\delta_u(x^{-1}y,e)=\delta_u(x,y)=d(x,y)<\varepsilon.$ This implies
that $x^{-1}y\in B_{\delta_u}(e,\varepsilon)$  and therefore
$j_2(\mathcal{E})\subseteq B_{\delta_u}(e,\varepsilon).$

We now show that $B_{\delta_u}(e,\varepsilon)\subseteq
[\widetilde{\mathcal{E}}].$ Let $e\neq w\in
B_{\delta_u}(e,\varepsilon),$ then by the definition of $\delta_u$
there exist words $$w^{\ast}=x_0\cdots x_n, \ v=y_0\cdots y_n\in
W(X)$$ such that $w=(w^{\ast})', v'=e$ and $d(x_i,y_i)<\varepsilon \
\forall i.$ We prove using induction on $lh(w^{\ast})=lh(v),$ that
$w\in [\widetilde{\mathcal{E}}].$ For $lh(w^{\ast})=1$ the assertion
holds trivially. For $lh(w^{\ast})=2$ assume that
$d(x_0,y_0)<\varepsilon, d(x_1,y_1)<\varepsilon$ and $y_1=y_0^{-1}.$
Then $d(x_1,y_1)=d(x_1^{-1},y_0)$ and since $d(x_0,y_0)<\varepsilon$
we obtain, using the strong triangle inequality, that
$d(x_0^{-1},x_1)=d(x_0,x_1^{-1})<\varepsilon.$ Since $\varepsilon<1$
and $x_0\neq x_1^{-1}$ it follows that $(x_0^{-1},x_1)\in X\times X$
or $(x_0^{-1},x_1)\in X^{-1}\times X^{-1}.$ In the first case
$(x_0^{-1},x_1)\in \mathcal{E}$ and thus $w=x_0x_1\in
j_2(\mathcal{E})\subseteq [\widetilde{\mathcal{E}}].$ In the second
case $(x_0,x_1^{-1})\in \mathcal{E}$
  and
thus $w=x_0x_1\in j_{2}^{\ast}(\mathcal{E})\subseteq
[\widetilde{\mathcal{E}}].$ Now assume the assertion is true for
$k<lh(w^{\ast})$ and that $lh(w^{\ast})\geq 3.$

First case: $y_0\neq y_n^{-1}.$ There exists $n>m$ such that
$y_0\cdots y_m=y_{m+1}\cdots y_n=e.$ By the induction hypothesis
$$x_0\cdots x_m, x_{m+1}\cdots x_n \in [\widetilde{\mathcal{E}}].$$
Since $[\widetilde{\mathcal{E}}]$ is a subgroup we have
$w\in [\widetilde{\mathcal{E}}].$\\
Second case: $y_0=y_n^{-1}.$ In this case $y_1\cdots y_{n-1}=e$ and
by the induction hypothesis $x_1\cdots x_{n-1}\in
[\widetilde{\mathcal{E}}].$ Since $[\widetilde{\mathcal{E}}]$ is
normal, $x_n^{-1}x_1\cdots x_{n-1}x_n\in [\widetilde{\mathcal{E}}].$
Since $y_0y_n=e$ it follows  from the induction hypothesis (for
$lh(w^{\ast})=2$)  that $x_0x_n\in [\widetilde{\mathcal{E}}].$
Finally, since $[\widetilde{\mathcal{E}}]$ is a subgroup,
$(x_0x_n)x_n^{-1}x_1\cdots x_{n-1}x_n=w\in
[\widetilde{\mathcal{E}}]$. This completes the proof of $(1)$.
\\
$(2):$ Immediately follows from $(1)$ and Theorem \ref{thm:nafin}.1
\end{proof}
\section{Comparison between Graev type ultra-metrics}  \label{sub:zar}  In   \cite{SZ} Savchenko and Zarichnyi introduced an ultra-metrization
$\hat{d}$ of the free group  over an ultra-metric space $(X,d)$ with
$diam(X)\leq 1.$  They used this ultra-metrization to study a
functor on the category of ultra-metric spaces of diameter$\leq 1$
and nonexpanding maps.

 Let $(X,d)$ be an ultra-metric space with diameter$\leq 1$.
Extend $d$ to  an ultra-metric on  $\overline{X}$  by defining
$$d(x^{-1},y^{-1})=d(x,y), \ d(x^{-1},y)=d(x,y^{-1})=d(x,e)=d(x^{-1},e)=1$$ for every $x,y\in X.$ Consider
its  associated Graev ultra-metric $\delta_u.$
  Our aim is to show that   $\delta_u=\hat{d}$ (Theorem
  \ref{thm:coinc}).
We first  provide the  definition of $\hat{d}$ from
 \cite{SZ}.

 Let $\alpha: F(X)\to \mathbb Z$ be the
continuous homomorphism extending the constant map $X\to
\{1\}\subseteq \mathbb Z.$ For every $r>0$ let $\mathcal{F}_{r}$ be
the partition of $X$ formed by the open balls with radius $r$ and
$q_r:X\to X/{F}_{r}$ is the quotient map.  Let $F(q_r):F(X)\to
F(X/{F}_{r})$ be the extension of $q_r:X\to X/{F}_{r}\hookrightarrow
F(X/{F}_{r}).$
 \begin{defi} (\cite[page 726]{SZ}) The function $\hat{d}:F(X)\times F(X)\to \mathbb R$ is defined as follows:

$$\hat{d}(v,w)= \left\{
   \begin{array}{ll}
     1, & \text{if}  \ \alpha(v)\neq\alpha(w) \\
     \inf\{r>0| \ F(q_r)(v)=F(q_r)(w)\} , & \text{if} \ \alpha(v)=\alpha(w)
   \end{array}
 \right.$$ for $v,w\in F(X).$
\end{defi}

\begin{thm}\label{thm:saz}
\cite[Theorem 3.1]{SZ} The function $\hat{d}$ is an invariant
continuous ultra-metric on the topological group $F(X).$
\end{thm}

\begin{lemma}\label{lem:geq}
For every $v,w\in F(X)$ we have $\delta_u(v,w)\geq \hat{d}(v,w).$
\end{lemma}
\begin{proof}
By Theorem \ref{thm:saz} and  Theorem \ref{theorem:grau}.b  it
suffices to prove that $\hat{d}$ extends the ultra-metric $d$
defined on $X\cup \{e\}$. For every $x\in X,  \ \alpha(x)=1\neq
0=\alpha(e).$ Thus
  for every $x\in X$ we have $\hat{d}(x,e)=d(x,e)=1.$  Let $x,y\in X.$
We have to show that $\hat{d}(x,y)=d(x,y).$

 Clearly $\alpha(x)=\alpha(y)=1.$ Therefore,
  $$\hat{d}(x,y)= \inf\{r>0| \ F(q_r)(x)=F(q_r)(y)\}= \inf\{r>0| \ q_r(x)=q_r(y)\}.$$

Denote $d(x,y)=s.$ It follows that $q_s(x)\neq q_s(y)$ and  for
every $r>s, \ q_r(x)= q_r(y).$ This implies that $$\inf\{r>0| \
q_r(x)=q_r(y)\}=s=d(x,y).$$

Hence  $\hat{d}(x,y)=d(x,y),$ which completes the proof.
\end{proof}

\begin{lemma}\cite[Lemma 3.5]{DG}\label{theorem:mat}
For any trivial word $w=x_0\cdots x_n$ there is a match $\theta$
such that for any $ i\leq n, \ x_{\theta(i)}=x_{i}^{-1}.$
\end{lemma}

\begin{lemma}\label{lem:leq}
For every $v,w\in F(X)$ we have $\delta_u(v,w)\leq \hat{d}(v,w).$
\end{lemma}

\begin{proof}
According to Theorems \ref{thm:gao}.2 and \ref{thm:saz}  both
 $\hat{d}$ and $\delta_u$ are invariant ultra-metrics. Therefore it
suffices to show that
$$\forall e\neq v\in F(X), \ \delta_u(v,e)\leq \hat{d}(v,e).$$

Let $v=x_0\cdots x_n\in F(X).$  Clearly $\delta_u(v,e)\leq 1.$ Thus
we may assume that $\alpha(v)=\alpha(e).$ Assume that $s>0$
satisfies $F(q_s)(v)=F(q_s)(e).$ We are going to show that there
exists a match $\theta$ such that $\rho(v,v^{\theta})<s.$ Using the
definition of $\hat{d}$ and Theorem \ref{thm:gao}.1 this will imply
that $\delta(v,e)\leq \hat{d}(v,e).$ For every $0\leq i \leq n$ let
$\overline{x_i}=F(q_s)(x_i).$ The equality $F(q_s)(v)=F(q_s)(e)$
suggests  that $\overline{x_0}\cdots \overline{x_n}\in W(X/{F}_{s})$
is a trivial word. By Lemma \ref{theorem:mat} there exists a match
$\theta$ such that for any $ i\leq n, \
\overline{x_{\theta(i)}}=\overline{x_{i}}^{-1}.$ Observe that
$\theta$ does not have fixed points. Indeed if $j$ is a fixed point
of $\theta$ then from the equalities
$\overline{x_{\theta(j)}}=\overline{x_{j}}^{-1}$ and
$\overline{x_{\theta(j)}}=\overline{x_j}$ we obtain that
$\overline{x_j}$ is the identity element of  $F(X/{F}_{s}).$ This
contradicts the fact that
 $x_j$ is not the identity element of
$F(X)$ and that $F(X/{F}_{s})$ is algebraically free over
$X/{F}_{s}$.

For every $0\leq i\leq n$ we conclude from the equality
$\overline{x_{\theta(i)}}=\overline{x_{i}}^{-1}$  that
$d(x_{i}^{-1},x_{\theta(i)})<s.$

Since $\theta$ does not have fixed points we obtain that
$$\rho_u(v,v^{\theta})=\max\{d(x_{i}^{-1},x_{\theta(i)}):\
\theta(i)<i\}<s.$$

This completes the proof.
\end{proof}

We finally obtain:
\begin{thm}\label{thm:coinc}
$\delta_u=\hat{d}$
\end{thm}
\begin{proof}
Use Lemma \ref{lem:geq} and Lemma \ref{lem:leq}.
\end{proof}

\noindent \textbf{Acknowledgment:} I would like to thank   M.
Megrelishvili  and L. Polev for  their useful suggestions.

\bibliographystyle{plain}

\end{document}